\pdfoutput=1
\documentclass[reqno]{shinyart}

\usepackage{algorithmic}
\usepackage[utf8]{inputenc}

\usepackage{shinybib}

\def\A{{{A}}}
\def\B{{{B}}}
\def\C{{{C}}}
\def\K{{{K}}}
\def\U{{{U}}}
\def\P{{{P}}}
\def\J{{{J}}}
\def\M{{{M}}}
\def\a{{{a}}}
\def\b{{{b}}}
\def\f{{{f}}}
\def\g{{{g}}}
\def\u{{{u}}}
\def\v{{{v}}}
\def\x{{{x}}}
\def\y{{{y}}}
\def\q{{{q}}}
\def\phii{{\varphi}}
\def\alphaa{{\alpha}}
\def\betaa{{\beta}}
\def\gammaa{{\gamma}}
\def\psii{{\psi}}
\def\etaa{{\eta}}

\def\R{\mathbb{R}}
\def\N{\mathbb{N}}
\def\m{m}
\def\n{n}
\def\one{\mathbb{1}}
\def\e{\mathrm{e}}
\renewcommand{\epsilon}{\varepsilon}
\newcommand{\norm}[1]{\Vert#1\Vert}
\newcommand{\abs}[1]{\vert#1\vert}

\DeclareMathOperator{\diag}{Diag}
\DeclareMathOperator{\spann}{span}

\usepackage{tikz,pgfplots}
\usepgfplotslibrary{colorbrewer}
\pgfplotsset{compat=newest}
\pgfplotsset{plot coordinates/math parser=false}
\pgfplotsset{every axis legend/.append style={legend cell align=left,align=left,draw=none,font=\footnotesize}}
\pgfplotsset{every axis label/.append style={font=\footnotesize}}
\pgfplotsset{every tick label/.append style={font=\tiny}}

\title{A Sinkhorn--Newton method for entropic optimal transport}

\author{
    Christoph Brauer\thanks{%
        Institute of Analysis and Algebra,
        TU Braunschweig,
        38092 Braunschweig, Germany
    (\email{ch.brauer@tu-braunschweig.de}, \email{d.lorenz@tu-braunschweig.de})}
    \and
    Christian Clason\thanks{%
        Faculty of Mathematics,
        University Duisburg-Essen,
        45117 Essen, Germany
    (\email{christian.clason@uni-due.de})}
    \and
    Dirk Lorenz\footnotemark[1]
    \and
    Benedikt Wirth\thanks{%
        Institute for Numerical and Applied Mathematics,
        University of Münster,
        Einsteinstraße 62, 48149 Münster, Germany
    (\email{benedikt.wirth@uni-muenster.de})}
}

\date{February 2, 2018}

\hypersetup{
    pdftitle={A Sinkhorn--Newton method for entropic optimal transport},
    pdfauthor={Christoph Brauer, Christian Clason, Dirk Lorenz, Benedikt Wirth},
    pdfkeywords={optimal transport, entropic regularization, Sinkhorn-Knopp, Newton method}
}
\begin{document}

\maketitle

\begin{abstract}
    We consider the entropic regularization of discretized optimal transport and propose to solve its optimality conditions via a logarithmic Newton iteration. We show a quadratic convergence rate and validate numerically that the method compares favorably with the more commonly used Sinkhorn--Knopp algorithm for small regularization strength. We further investigate numerically the robustness of the proposed method with respect to parameters such as the mesh size of the discretization.
\end{abstract}

\section{Introduction}
\label{sec:introduction}

The mathematical problem of optimal mass transport has a long history dating back to its introduction in \textcite{Monge1781}, with key contributions by \textcite{Kantorovich1942} and \textcite{Kantorovich1957}. It has recently received increased interest due to numerous applications in machine learning; see, e.g., the recent overview of \textcite{Kolouri2017} and the references therein. In a nutshell, the (discrete) problem of optimal transport in its Kantorovich form is to compute for given mass distributions $\a$ and $\b$ with equal mass a transport plan, i.e., an assignment of how much mass of $\a$ at some point should be moved to another point to match the mass in $\b$. This should be done in a way such that some transport cost (usually proportional to the amount of mass and dependent on the distance) is minimized. This leads to a linear optimization problem which has been well studied, but its application in machine learning has been problematic due to large memory requirement and long run time. Recently, \textcite{Cuturi2013} proposed a method that overcomes the memory requirement by so-called entropic regularization that has found broad applications; see, e.g., \textcite{Carlier2017, Cuturi2014, Frogner2015}. The resulting iteration resembles the so-called Sinkhorn--Knopp method from \textcite{Sinkhorn1967} for matrix balancing and allows for a simple and efficient implementation.

\subsection{Our contribution}
\label{sec:contribution}

In this work, we show that the Sinkhorn--Knopp method can be viewed as an approximate Newton method and derive a full Newton method for entropically regularized optimal transport problems that is demonstrated to perform significantly better for small entropic regularization parameters. Here, compared to \textcite{Cuturi2013}, the key idea is to apply a logarithmic transform to the variables.

This paper is organized as follows. In \cref{sec:sinkhorn_newton_method}, we state the Kantorovich formulation of optimal transport together with its dual which serves as the basis of the derived algorithm. Afterwards, we establish local quadratic convergence and discuss the relation of the proposed Newton method to the Sinkhorn--Knopp iteration. The performance and parameter dependence of the proposed method are illustrated with numerical examples in \cref{sec:numerical_examples}. \Cref{sec:proof} contains the proof of the key estimate for quadratic convergence, and \cref{sec:conclusion} concludes the paper.

\subsection{Notation}
\label{sec:notation}

In the following, $\one_{\n}$ represents the $\n$-dimensional vector with all ones and $\one_{\n, \m}$ refers to the $\n\times \m$ matrix with all ones. Moreover, $\Sigma_{\n} \coloneqq \{\a\in\R^{\n}_{+} : \one_{\n}^{\top}\a = 1\}$ denotes the probability simplex in $\R^{\n}_{+}$ whose elements are called \emph{probability vectors}, or equivalently, \emph{histograms}. For two histograms $\a\in\Sigma_{\n}$ and $\b\in\Sigma_{m}$,
\begin{equation}
    \label{eq:admissible_transport_plans}
    \U(\a, \b) \coloneqq \{\P\in\R^{\n\times \m}_{+} : \P\one_{\m} = \a, \ \P^{\top}\one_{\n} = \b\}
\end{equation}
is the set of admissible \emph{coupling matrices}. In the context of optimal transport, the elements of $\U(\a, \b)$ are also referred to as \emph{transport plans}.
Histograms $\a$ and $\b$ can be viewed as mass distributions,
and an entry $\P_{ij}$ of a transport plan $\P\in\U(\a, \b)$ can be interpreted as the amount of mass moved from $\a_i$ to $\b_j$.

We refer to the Frobenius inner product of two matrices $\P, \P'\in\R^{\n\times \m}$ as $\langle \P, \P' \rangle \coloneqq \sum_{ij}\P_{ij}\P'_{ij}$. At the same time, $\langle\a, \a'\rangle \coloneqq \sum_{i}\a_{i}\a_{i}'$ denotes the standard dot product of two vectors $\a, \a'\in\R^{\n}$. Finally, $\diag(\a)\in\R^{\n\times \n}$ is defined as the diagonal matrix with $\diag(\a)_{ii} \coloneqq \a_{i}$ and $\diag(\a)_{ij} \coloneqq 0$ for $i \neq j$, and $\a\odot\a'\coloneqq \diag(\a)\a'$ is the Hadamard product (i.e., the component-wise product) of $\a$ and $\a'$.

\section{Sinkhorn--Newton method}
\label{sec:sinkhorn_newton_method}

In this section we derive our Sinkhorn--Newton method. We start by introducing the problem of entropically regularized optimal transport in \cref{sec:problem_setting}. Afterwards, in \cref{sec:algorithm}, we present our approach, which is essentially applying Newton's method to the optimality system associated with the transport problem and its dual, before we discuss its local quadratic convergence in \cref{sec:convergence}. In \cref{sec:relation_sinkhorn_knopp}, we finally establish a connection between our Newton iteration and the Sinkhorn--Knopp type iteration introduced by \textcite{Cuturi2013}.

\subsection{Problem setting}
\label{sec:problem_setting}

Let $\a\in\Sigma_{\n}$ and $\b\in\Sigma_{\m}$ be given histograms together with a non-negative cost matrix $\C\in\R^{\n\times \m}$. The entropically regularized Kantorovich problem of optimal mass transport between $\a$ and $\b$ is
\begin{equation}
    \label{eq:kantorovich_entropically}
    \tag{P$_{\epsilon}$}
    \inf_{\P \in \U(\a, \b)} \ \langle\C, \P\rangle + \epsilon \langle\P, \log\P - \one_{\n,\m}\rangle,
\end{equation}
where the logarithm is applied componentwise to $\P$ and $\epsilon > 0$ is the regularization strength. The variables $\P_{ij}$ indicate how much of $\a_{i}$ ends up in $\b_{j}$, while $\C_{ij}$ is the corresponding transport cost per unit mass. Abbreviating $\K\coloneqq \exp(-\C / \epsilon)$, standard convex duality theory leads us to the dual problem
\begin{equation}
    \label{eq:dual_problem}
    \tag{D$_{\epsilon}$}
    \sup_{\f\in\R^{\n}, \g\in\R^{\m}} - \langle\a, \f\rangle - \langle\b, \g\rangle - \epsilon \langle \e^{-\f / \epsilon}, \K \e^{-\g / \epsilon}\rangle,
\end{equation}
where $\f$ and $\g$ are the dual variables and the exponential function is applied componentwise. The problems \eqref{eq:kantorovich_entropically} and \eqref{eq:dual_problem} are linked via the optimality conditions
\begin{subequations}
    \begin{align}
        \P &= \diag(\e^{-\f / \epsilon})\K\diag(\e^{-\g / \epsilon})\label{eq:oc_primal_dual}\\
        \a &= \diag(\e^{-\f / \epsilon})\K\e^{-\g / \epsilon}\label{eq:oc_source_mass_conservation}\\
        \b &= \diag(\e^{-\g / \epsilon})\K^{\top}\e^{-\f / \epsilon}\label{eq:oc_sink_mass_conservation}.
    \end{align}
    \label{eq:oc}%
\end{subequations}
The first condition \eqref{eq:oc_primal_dual} connects the optimal transport plan with the dual variables. The conditions \eqref{eq:oc_source_mass_conservation} and \eqref{eq:oc_sink_mass_conservation} simply reflect the feasibility of $\P$ for \eqref{eq:kantorovich_entropically}, i.e., for the mass conservation constraints in \eqref{eq:admissible_transport_plans}.

\subsection{Algorithm}
\label{sec:algorithm}

Finding dual vectors $\f$ and $\g$ that satisfy \eqref{eq:oc_source_mass_conservation} and \eqref{eq:oc_sink_mass_conservation} is equivalent to finding a root of the function
\begin{equation}
    \label{eq:oc_function}
    F(\f, \g) \coloneqq
    \begin{pmatrix*}[l]
        \a - \diag(\e^{-\f / \epsilon})\K\e^{-\g / \epsilon}\\
        \b - \diag(\e^{-\g / \epsilon})\K^{\top}\e^{-\f / \epsilon}
    \end{pmatrix*},
\end{equation}
i.e., to solving $F(\f, \g) =  0$. A Newton iteration for this equation is given by
\begin{equation}
    \label{eq:newton_iteration}
    \begin{pmatrix}
        \f^{k+1}\\ \g^{k+1}
    \end{pmatrix}
    =
    \begin{pmatrix}
        \f^{k}\\ \g^{k}
    \end{pmatrix}
    -
    \J_{F}(\f^{k}, \g^{k})^{-1}F(\f^{k}, \g^{k}).
\end{equation}
The Jacobian matrix of $F$ is
\begin{equation}
    \label{eq:jacobian}
    \J_{F}(\f, \g) = \frac{1}{\epsilon}
    \begin{bmatrix}
        \diag(\P\one_{m}) &\P\\
        \P^{\top} &\diag(\P^{\top}\one_{\n})
    \end{bmatrix},
\end{equation}
where we used \eqref{eq:oc_primal_dual} to simplify the notation. Performing the Newton step \eqref{eq:newton_iteration} requires finding a solution of the linear equation system
\begin{equation}
    \label{eq:newton_step_les}
    \J_{F}(\f^{k}, \g^{k})\left(
        \begin{matrix}
            \delta\f\\ \delta\g
    \end{matrix}\right)
    = -F(\f^{k}, \g^{k}).
\end{equation}
The new iterates are then given by
\begin{subequations}
    \begin{align}
        \f^{k+1} &= \f^{k} + \delta\f\\
        \g^{k+1} &= \g^{k} + \delta\g.
    \end{align}
    \label{eq:dual_variables_update}%
\end{subequations}
If one is only interested in the optimal transport plan, then it is actually not necessary to keep track of the dual iterates $\f^{k}$ and $\g^{k}$ after initialization (in our subsequent experiments, we use $\f^{0} = \g^{0} =  0$ and hence, $\P^{0} = \K$). This is true because \eqref{eq:newton_step_les} can be expressed entirely in terms of
\begin{equation}
    \label{eq:p_k}
    \P^{k} \coloneqq \diag(\e^{-\f^{k} / \epsilon})\K\diag(\e^{-\g^{k} / \epsilon}),
\end{equation}
and thus, using \eqref{eq:dual_variables_update} and \eqref{eq:p_k}, we obtain the multiplicative update rule
\begin{equation}
    \label{eq:primal_variables_update}
    \begin{aligned}[b]
        \P^{k + 1} &= \diag(\e^{-[\f^{k} + \delta\f] / \epsilon})\K\diag(\e^{-[\g^{k} + \delta\g] / \epsilon})\\
                   &= \diag(\e^{-\delta\f / \epsilon})\P^{k}\diag(\e^{-\delta\g / \epsilon}).
    \end{aligned}
\end{equation}
In this way, we obtain an algorithm which only operates with primal variables, see \cref{alg:sinkhorn-newton-primal}.
In applications where the storage demand for the plans $\P^{k}$ is too high and one is only interested in the optimal value, there is another form which does not form the plans $\P^{k}$, but only the dual variables $\f^{k}$ and $\g^{k}$ and which can basically operate matrix-free. We sketch it as \cref{alg:sinkhorn-newton-dual} below.

\begin{algorithm}[tb]
    \caption{Sinkhorn-Newton method in primal variable}\label{alg:sinkhorn-newton-primal}
    \begin{algorithmic}[1]
        \STATE \textbf{Input:}  $\a\in\Sigma_{n}$, $\b\in\Sigma_{m}$, $\C\in\R^{n\times m}$
        \STATE \textbf{Initialize:} $\P^{0} = \exp(-\C /\epsilon)$, set $k=0$ 
        \REPEAT
        \STATE Compute approximate histograms
        \[
            \a^{k} = \P^{k}\one_{m},\quad \b^{k} = (\P^{k})^{\top}\one_{n}.
        \]
        \STATE Compute updates $\delta\f$ and $\delta\g$ by solving
        \[
            \frac{1}{\epsilon}
            \begin{bmatrix}
                \diag(\a^{k}) &\P^{k}\\
                (\P^{k})^{\top} &\diag(\b^{k})
            \end{bmatrix}
            \begin{bmatrix}
                \delta\f\\\delta\g
            \end{bmatrix}
            = 
            \begin{bmatrix}
                \a^{k}-\a\\
                \b^{k}-\b
            \end{bmatrix}.
        \]
        \STATE Update $\P$ by
        \[
            \P^{k + 1} = \diag(\e^{-\delta\f / \epsilon})\P^{k}\diag(\e^{-\delta\g / \epsilon}).
        \]
        \STATE $k\leftarrow k+1$
        \UNTIL{some stopping criteria fulfilled}
    \end{algorithmic}
\end{algorithm}

\begin{algorithm}[tb]
    \caption{Sinkhorn-Newton method in dual variables}\label{alg:sinkhorn-newton-dual}
    \begin{algorithmic}[1]
        \STATE \textbf{Input:}  $\a\in\Sigma_{n}$, $\b\in\Sigma_{m}$, function handle for application of $\K$ and $\K^{\top}$
        \STATE \textbf{Initialize:} $\a^{0}\in\R^{n}$, $\b^{0}\in\R^{m}$, set $k=0$
        \REPEAT
        \STATE Compute approximate histograms
        \[
            \a^{k} = \e^{-\f^{k} / \epsilon}\odot \K\e^{-\g^{k} /\epsilon},\quad \b^{k} = \e^{-\g^{k} / \epsilon}\odot \K^{\top}\e^{-\f^{k} /\epsilon}.
        \]
        \STATE Compute updates $\delta\f$ and $\delta\g$ by solving
        \[\M
            \begin{bmatrix}
                \delta\f\\\delta\g
            \end{bmatrix}
            = 
            \begin{bmatrix}
                \a^{k}-\a\\
                \b^{k}-\b
            \end{bmatrix}
        \]
        where the application of $\M$ is given by
        \[
            \M
            \begin{bmatrix}
                \delta\f\\\delta\g
            \end{bmatrix}
            = 
            \frac1\epsilon
            \begin{bmatrix}
                \a^{k}\odot\delta\f + \e^{-\f^{k}/\epsilon}\odot \K(\e^{-\g^{k}/\epsilon}\odot\delta\g)\\
                \b^{k}\odot\delta\g + e^{-\g^{k}/\epsilon}\odot \K^{\top}(\e^{-\f^{k}/\epsilon}\odot\delta\f)
            \end{bmatrix}.
        \]
        \STATE Update $\f$ and $\g$ by
        \[
            \f^{k + 1} = \f^{k} + \delta\f,\quad \g^{k + 1} = \g^{k} + \delta\g.
        \]
        \STATE $k\leftarrow k+1$
        \UNTIL{some stopping criteria fulfilled}
    \end{algorithmic}
\end{algorithm}

\subsection{Convergence and numerical aspects}
\label{sec:convergence}

In the following, we first argue that \eqref{eq:newton_step_les} is solvable. Then we show that the sequence of Newton iterates converges locally at a quadratic rate as long as the optimal transport plan satisfies $\P \geq c\cdot\one_{\n, \m}$ for some constant $c > 0$.

\begin{lemma}
    \label{lem:jacobian_sym_pos_def}
    For $\f\in\R^{\n}$ and $\g\in\R^{\m}$, the Jacobian matrix $\J_{F}(\f, \g)$ is symmetric positive semi-definite, and its kernel is given by
    \begin{equation}
        \label{eq:jacobian_kernel}
        \ker \left[\J_{F}(\f, \g)\right] = \spann\left\{
            \left(\begin{matrix*}[l]
                    \hphantom{-}\one_{\n}\\ -\one_{m}
            \end{matrix*}\right)      
        \right\}.
    \end{equation}
\end{lemma}
\begin{proof}
    The matrix is obviously symmetric. For arbitrary $\phii\in\R^{\n}$ and $\gammaa\in\R^{\m}$, we obtain from \eqref{eq:jacobian} that
    \begin{equation}
        \label{eq:jacobian_pos_semi_def}
        \left(\begin{matrix}
                \phii^{\top} &\gammaa^{\top}
        \end{matrix}\right)
        \J_{F}(\f, \g)
        \left(\begin{matrix}
                \phii\\ \gammaa
        \end{matrix}\right)
        =
        \frac{1}{\epsilon}\sum_{ij}\P_{ij}(\phii_{i} + \gammaa_{j})^{2} \geq 0,
    \end{equation}
    which holds with equality if and only if we have $\phii_{i} + \gammaa_{j} = 0$ for all $i, j$.
\end{proof}
Hence, the system \eqref{eq:newton_step_les} can be solved by a conjugate gradient (CG) method. To see that, recall that the CG method iterates on the orthogonal complement of the kernel as long as the initial iterate $(\delta\f^0,\delta\g^0)$ is chosen from this subspace, in this case with $\one_{\n}^{\top}\delta\f^{0} = \one_{m}^{\top}\delta\g^{0}$. Furthermore, the Newton matrix can be applied matrix-free in an efficient manner as soon as the multiplication with $\K = \exp(-\C/\epsilon)$ and its transpose can be done efficiently, see \cref{alg:sinkhorn-newton-dual}. This is the case, for example if $\C_{ij}$ only depends on $i-j$ and thus, multiplication with $\K$ amounts to a convolution. A cheap diagonal preconditioner is provided by the matrix
\begin{equation}
    \label{eq:preconditioner}
    \frac{1}{\epsilon}
    \begin{bmatrix}
        \diag(\P^{k}\one_{\n}) & 0\\
        0 &\diag([\P^{k}]^{\top}\one_{m})
    \end{bmatrix}.
\end{equation}

According to \textcite[Thm.~2.3]{Deuflhard2011}, we expect local quadratic convergence as long as
\begin{equation}
    \label{eq:newton_condition}
    \norm{\J_{F}(\y^{k})^{-1}[\J_{F}(\y^{k}) - \J_{F}(\etaa)](\y^{k} - \etaa)} \leq \omega \norm{\y^{k} - \etaa}^{2}
\end{equation}
holds for all $\etaa\in\R^{\n}\times\R^{m}$ and $k\in\N$, with an arbitrary norm and some constant $\omega > 0$ in a neighborhood of the solution. Here, we abbreviated $\y^{k} \coloneqq (\f^{k}, \g^{k})$.

\begin{theorem}
    \label{thm:newton_constant}
    For any $k\in\N$ with $\P_{ij}^{k} > 0$, \eqref{eq:newton_condition} holds in the $\ell_{\infty}$-norm for
    \begin{equation}
        \label{eq:newton_constant}
        \omega \leq (\e^{\frac1\epsilon}-1)\left(1 + 2\e^{\frac1\epsilon}\frac{\max\big\{\norm{\P^{k}\one_{m}}_{\infty}, \, \norm{[\P^{k}]^{\top}\one_{n}}_{\infty}\big\}}{\min_{ij}\P_{ij}^{k}}\right)
    \end{equation}
    when $\norm{\y^{k}-\etaa}_{\infty}\leq 1$.
\end{theorem}
We postpone the proof of \cref{thm:newton_constant} to \cref{sec:proof}.
\begin{remark}
    In fact, one can show that necessarily $\omega\geq\e^{\frac1\epsilon}-1$.
    Indeed, if $\y^k-\eta=(\phii,0)\in\R^n\times\R^n$, then one can explicitly compute
    \begin{equation*}
        \J_{F}(\y^{k})^{-1}[\J_{F}(\y^{k}) - \J_{F}(\etaa)](\y^{k} - \etaa)
        =((\e^{\phii/\epsilon}-1)\phii,0),
    \end{equation*}
    where the exponential and the multiplication are pointwise (the calculation is detailed in the proof of \cref{thm:newton_constant}).
\end{remark}

Hence, if $(\f^{0}, \g^{0})$ is chosen sufficiently close to a solution of $F(\f, \g) =  0$, then the contraction property of Newton's method shows that the sequence of Newton iterates $(\f^{k}, \g^{k})$, and hence $\P^{k}$, remain bounded. If the optimal plan satisfies $\P^{*} \geq c\cdot\one_{\n, \m}$ for some $c > 0$, we can therefore expect local quadractic convergence of Newton's method.

\subsection{Relation to Sinkhorn--Knopp}
\label{sec:relation_sinkhorn_knopp}

Substituting $\u\coloneqq \e^{-\f / \epsilon}$ and $\v \coloneqq \e^{-\g / \epsilon}$ in \eqref{eq:oc} shows that the optimality system can be written equivalently as
\begin{subequations}
    \begin{align}
        \P &= \diag(\u)\K\diag(\v)\label{eq:oc_primal_dual_subst}\\
        \a &= \diag(\u)\K\v\label{eq:oc_source_mass_conservation_subst}\\
        \b &= \diag(\v)\K^{\top}\u\label{eq:oc_sink_mass_conservation_subst}.
    \end{align}
\end{subequations}
In order to find a solution of \eqref{eq:oc_source_mass_conservation_subst}--\eqref{eq:oc_sink_mass_conservation_subst}, one can apply the Sinkhorn--Knopp algorithm \cite{Sinkhorn1967} as recently proposed in \textcite{Cuturi2013}. This amounts to alternating updates in the form of
\begin{subequations}
    \begin{align}
        \u^{k+1} &\coloneqq \diag(\K\v^{k})^{-1}\a\label{eq:sinkhorn_1}\\
        \v^{k+1} &\coloneqq \diag(\K^{\top}\u^{k+1})^{-1}\b.\label{eq:sinkhorn_2}
    \end{align}
    \label{eq:sinkhorn_updates}%
\end{subequations}
In \eqref{eq:sinkhorn_1}, $\u^{k+1}$ is updated such that $\u^{k+1}$ and $\v^{k}$ solve \eqref{eq:oc_source_mass_conservation_subst}, and in the subsequent \eqref{eq:sinkhorn_2}, $\v^{k+1}$ is updated such that $\u^{k+1}$ and $\v^{k+1}$ form a solution of \eqref{eq:oc_sink_mass_conservation_subst}.

If we proceed analogously to \cref{sec:algorithm} and derive a Newton iteration to find a root of the function
\begin{equation}
    \label{eq:oc_function_subst}
    G(\u, \v) \coloneqq
    \begin{pmatrix*}[l]
        \diag(\u)\K\v - \a\\
        \diag(\v)\K^{\top}\u - \b
    \end{pmatrix*},
\end{equation}
then the associated Jacobian matrix is
\begin{equation}
    \label{eq:oc_function_jacobian_subst}
    \J_{G}(\u, \v) =
    \begin{pmatrix*}[l]
        \diag(\K\v) &\diag(\u)\K\\
        \diag(\v)\K^{\top} &\diag(\K^{\top}\u)
    \end{pmatrix*}.
\end{equation}
Neglecting the off-diagonal blocks in \eqref{eq:oc_function_jacobian_subst} and using the approximation
\begin{equation}
    \label{eq:oc_function_jacobian_subst_approx}
    \hat\J_{G}(\u, \v) =
    \begin{pmatrix*}[l]
        \diag(\K\v) & 0\\
        0 &\diag(\K^{\top}\u)
    \end{pmatrix*}
\end{equation}
to perform the Newton iteration
\begin{equation}
    \label{eq:newton_iteration_subst}
    \begin{pmatrix}
        \u^{k+1}\\ \v^{k+1}
    \end{pmatrix}
    =
    \begin{pmatrix}
        \u^{k}\\ \v^{k}
    \end{pmatrix}
    -
    \hat\J_{G}(\u^{k}, \v^{k})^{-1}G(\u^{k}, \v^{k})
\end{equation}
leads us to the parallel updates
\begin{subequations}
    \begin{align}
        \u^{k+1} &\coloneqq \diag(\K\v^{k})^{-1}\a\label{eq:sinkhorn_newton_1}\\
        \v^{k+1} &\coloneqq \diag(\K^{\top}\u^{k})^{-1}\b.\label{eq:sinkhorn_newton_2}
    \end{align}
    \label{eq:sinkhorn_updates_newton}%
\end{subequations}
Hence, we see that a Sinkhorn--Knopp step \eqref{eq:sinkhorn_updates} simply approximates one Newton step \eqref{eq:sinkhorn_updates_newton} by neglecting the off-diagonal blocks and replacing $\u^{k}$ by $\u^{k+1}$ in \eqref{eq:sinkhorn_newton_2}. In our experience, neither the Newton iteration for $G(\u,\v)= 0$ (which seems to work for the less general problem of matrix balancing; see \textcite{Knight2013}) nor the version of Sinkhorn--Knopp in which $\v^{k+1}$ is updated using $\u^{k}$ instead of $\u^{k+1}$ converge.

\section{Numerical examples}
\label{sec:numerical_examples}

We illustrate the performance of the Sinkhorn--Newton method and its behavior by several examples. 
We note that a numerical comparison is not straightforward as there a several possibilities to tune the method, depending on the structure at hand and on the specific goals.
As illustrated in \cref{sec:algorithm}, one could take advantage of fast applications of the matrix $\K$ or use less memory if one does not want to store $\P$ during the iteration.
Here we focus on the comparison with the usual (linearly convergent) Sinkhorn iteration.
Thus, we do not aim for greatest overall speed but for a fair comparison between the Sinkhorn-Newton method and the Sinkhorn iteration.
To that end, we observe that one application of the Newton matrix~\eqref{eq:jacobian} amounts to one multiplication with $\P$ and $\P^{\top}$ each, two coordinate-wise products and sums of vectors.
For one Sinkhorn iteration we need one multiplication with $\K$ and $\K^{\top}$ and two additional coordinate-wise operations.
Although \cref{alg:sinkhorn-newton-dual} looks a little closer to the Sinkhorn iteration, we still compare \cref{alg:sinkhorn-newton-primal}, as we did not exploit any of the special structure in $\K$ or $\P$.

All timings are reported using MATLAB (R2017b) implementations of the methods on an Intel Xeon E3-1270v3 (four cores at 3.5\,GHz) with 16\,GB RAM.
The code used to generate the results below can be downloaded from \url{https://github.com/dirloren/sinkhornnewton}.

In all our experiments, we address the case $\m = \n$ and the considered histograms are defined on equidistant grids $\{\x_{i}\}_{i = 1}^{\n} \subset [0, 1]^{d}$ with $d = 2$ (in \cref{sec:comparison_sinkhorn,sec:dependence_regularization_strength}) and $d = 1$ (in \cref{sec:dependence_problem_dimension}), respectively. Throughout, the cost is chosen as quadratic, i.e., $\C_{ij} \coloneqq \norm{\x_{i} - \x_{j}}_{2}^{2}$.

Our Sinkhorn--Newton method is implemented according to \cref{alg:sinkhorn-newton-primal} and using a preconditioned CG method. The iteration is terminated as soon as the maximal violation of the constraints $\norm{\a^{k} - \a}_{\infty}$ and $\norm{\b^{k} - \b}_{\infty}$ drops below some threshold. If applicable, the same termination criterion is chosen for the Sinkhorn method, which is initialized with $\u^{0} = \v^{0} = \one_{\n}$.

\subsection{Comparison with Sinkhorn--Knopp}
\label{sec:comparison_sinkhorn}

\begin{figure*}
    \begin{subfigure}{0.45\linewidth}
        \centering
        \begin{tikzpicture}

\begin{axis}[%
width=0.9\linewidth,
scale only axis,
xmin=0,
xmax=3500,
ymode=log,
ymin=1e-15,
ymax=100000,
yminorticks=true,
xmajorgrids,
ymajorgrids,
yminorgrids,
]
\addplot [color=PuOr-C, line width=1.0,dashed]
  table[row sep=crcr]{%
0	1.26500023669347\\
1	0.0136416427633268\\
7	0.0130477466727975\\
16	0.0115975142181455\\
29	0.010524172558557\\
33	0.00990584470607129\\
39	0.00895581965221305\\
51	0.0078128561576014\\
57	0.00758150352848728\\
64	0.00729291858269186\\
69	0.00680704398762324\\
77	0.00608352911908308\\
83	0.00579223586191635\\
99	0.00513523594238887\\
108	0.00477804855556539\\
114	0.00468003438485491\\
121	0.00447109934150242\\
155	0.00322641484025231\\
217	0.00217984338884465\\
259	0.00153062708291259\\
325	0.000844768460026495\\
356	0.000668209712519447\\
397	0.000512106853887997\\
495	0.000273776074097261\\
559	0.000172159263331388\\
688	6.46677170053247e-05\\
3342	9.97344568043809e-14\\
};
\addlegendentry{S-viol.}

\addplot [color=PuOr-F, line width=1.0,dashed]
  table[row sep=crcr]{%
0	0.221003494943333\\
1	0.0738956499856481\\
15	0.0722247853198756\\
43	0.0665589153895944\\
88	0.0556145493846667\\
114	0.0487661643396954\\
179	0.03414389392809\\
224	0.0254571764031574\\
507	0.0035198877160952\\
636	0.00132990065544687\\
1353	5.5397459539926e-06\\
3168	5.2847587417317e-12\\
3313	1.80587489406735e-12\\
3342	1.4651335700226e-12\\
};
\addlegendentry{S--cost}

\addplot [color=PuOr-M, line width=1.0,dashed]
  table[row sep=crcr]{%
0	501.805343098149\\
1	1.959323781483\\
41	1.97265712315129\\
100	1.94704827403278\\
144	1.85813143255746\\
169	1.74541152301852\\
194	1.58617349532564\\
225	1.37246072047158\\
263	1.11300092726056\\
304	0.859373000714589\\
442	0.330730990871225\\
521	0.186171860696127\\
674	0.0584289698344776\\
1504	0.000102522774035328\\
3168	3.10299501531485e-10\\
3320	1.00767285718505e-10\\
3342	8.60047865880075e-11\\
};
\addlegendentry{S--plan}

\addplot [color=PuOr-C, line width=1.0]
  table[row sep=crcr]{%
0	1.26500023669347\\
33	0.465723221403546\\
65	0.173496284370927\\
97	0.0733382380440359\\
130	0.0484468942952844\\
163	0.0387764838574161\\
195	0.743423412274306\\
228	3.03156337899316\\
292	0.412020107039968\\
324	0.156251763186458\\
357	0.084722987060963\\
423	0.112188146700435\\
455	0.0504921085159309\\
486	0.0279719626954247\\
519	0.0125683914010298\\
552	0.00978639172922093\\
584	0.00366388372340697\\
616	0.0010886657361573\\
648	0.000156563789277246\\
681	5.45807768968397e-06\\
712	7.54259303426155e-07\\
745	4.02140594366584e-07\\
777	1.3425066051119e-07\\
809	6.26825376482577e-08\\
840	2.4729718672145e-09\\
871	3.45846972067742e-09\\
903	4.72850075067336e-11\\
932	4.39773247332207e-10\\
961	1.12774088678619e-10\\
990	4.09443676879697e-11\\
1019	4.07144352954063e-11\\
1050	2.33918266701101e-12\\
1078	6.16950240894734e-12\\
1106	2.08352085584496e-12\\
1133	6.35979116592248e-13\\
1160	2.22020318796366e-13\\
1191	1.21864324187387e-14\\
};
\addlegendentry{N--viol.}

\addplot [color=PuOr-F, line width=1.0]
  table[row sep=crcr]{%
0	0.221003494943333\\
33	0.0348027939902073\\
65	0.0324450618254047\\
97	0.053723027247762\\
130	0.0545036215481336\\
163	0.0512610239201841\\
195	0.102772252958704\\
228	0.303322300524973\\
260	0.125708045645722\\
292	0.0853457011809427\\
324	0.084762965958885\\
390	0.108346197557021\\
423	0.103566226993953\\
455	0.0928402218371696\\
486	0.0548938086334201\\
519	0.0270596944110355\\
552	0.017259615680877\\
584	0.00608022463560978\\
616	0.00113234166210819\\
648	0.000141088297629181\\
681	4.12439914371139e-06\\
712	2.84491281549229e-06\\
745	1.94043427451092e-07\\
777	1.73574878387069e-07\\
809	4.4446452504561e-08\\
840	1.41289910882164e-08\\
871	1.11696647347976e-08\\
903	4.16690723592493e-10\\
932	1.27511422254288e-09\\
961	4.35363384232551e-10\\
990	1.89982876652743e-10\\
1019	5.31707178286225e-11\\
1050	7.55363827042933e-12\\
1078	1.27680921391281e-11\\
1106	2.5399821135249e-12\\
1133	1.37644062814229e-12\\
1160	7.94961318995097e-13\\
};
\addlegendentry{N--cost}

\addplot [color=PuOr-M, line width=1.0]
  table[row sep=crcr]{%
0	501.805343098411\\
65	69.5566567347726\\
97	27.4713159668588\\
130	13.1636645671118\\
163	8.05222141295019\\
195	15.153552014504\\
228	21.8117959463902\\
260	10.9070790253686\\
292	6.99271452463659\\
324	5.39051956935373\\
390	3.65126160566395\\
423	2.85298467844127\\
455	1.9683556749751\\
519	0.596396581274534\\
552	0.289386003502393\\
584	0.0923500097805397\\
616	0.0197102233736058\\
648	0.00451037073923236\\
681	0.000759257738626059\\
712	0.000198629767513873\\
745	5.87452864986108e-05\\
777	1.85240999305121e-05\\
809	7.1334435951558e-06\\
840	1.15924256544295e-06\\
871	9.41789691491836e-07\\
903	3.25010421756952e-08\\
932	1.23734109597959e-07\\
961	3.05843731699112e-08\\
990	1.28928431935794e-08\\
1019	7.19546431176695e-09\\
1050	5.53296017708501e-10\\
1078	1.10538883974399e-09\\
1106	2.6692296070361e-10\\
1133	1.18818428720556e-10\\
1160	6.51427977827055e-11\\
};
\addlegendentry{N--plan}

\end{axis}
\end{tikzpicture}%
        \caption{Errors over iterations (CG for Newton).}
        \label{fig:errors:its}
    \end{subfigure}
    \hspace{1cm}
    \begin{subfigure}{0.45\linewidth}
        \centering
        \input{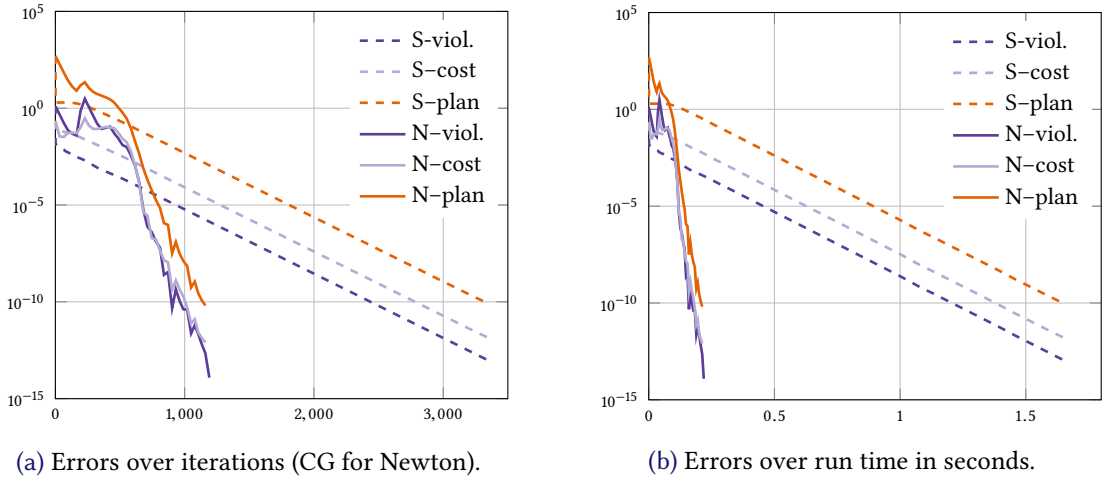}
        \caption{Errors over run time in seconds.}
        \label{fig:errors:time}
    \end{subfigure}
    \caption{Performance of Sinkhorn (S) and Newton (N) iterations measured by constraint violation (viol.), distance to optimal transport cost (cost) and distance to optimal transport plan (plan).}
    \label{fig:errors}
\end{figure*}

We first address the comparison of Sinkhorn--Newton with the classical Sinkhorn iteration. For this purpose, we discretize the unit square $[0, 1]^{2}$ using a $20\times 20$ equidistant grid $\{\x_{i} = (\x_{i1}, \x_{i2})\}_{i = 1}^{400}\subset [0, 1]^{2}$ and take
\begin{subequations}
    \begin{align}
        \label{eq:a_b_test}
        \tilde \a_{i} &\coloneqq \e^{-36([\x_{i1} - \frac13]^{2} - [\x_{i2} -\frac13]^{2})} + 10^{-1},\\
        \tilde \b_{j} &\coloneqq \e^{-9([\x_{j1} - \frac23]^{2} - [\x_{j2} -\frac23]^{2})} + 10^{-1},
    \end{align}
\end{subequations}
which are then normalized to unit mass by setting
\begin{equation}
    \label{eq:a_b_test_2}
    \a\coloneqq \frac{\tilde \a}{\sum_{i}\tilde\a_{i}} \quad \text{and} \quad \b\coloneqq \frac{\tilde \b}{\sum_{j}\tilde\b_{j}}.
\end{equation}
The entropic regularization parameter is set to $\epsilon \coloneqq 10^{-3}$ and in case of Sinkhorn-Newton, the CG method is implemented with a tolerance of $10^{-13}$ and a maximum number of $34$ iterations. Moreover, the threshold for the termination criterion is chosen as $10^{-13}$.

\Cref{fig:errors} shows the convergence history of the constraint violation for both iterations together with the error in the unregularized transport cost $\abs{\langle\C, \P^{k} - \P^{*}\rangle}$, where $\P^{*}$ denotes the final transport plan, and the error in the transport plan $\norm{\P^{k} - \P^{*}}_{1}$. In \cref{fig:errors:its}, we compare the error as a function of the iterations, where we take the total number of CG iterations for Sinkhorn--Newton to allow for a fair comparison (since both a Sinkhorn and a CG step have comparable costs, dominated by the two dense matrix--vector products $\K\v$, $\K^{\top}\u$ and $\P^{\top}\delta\f$, $\P\delta\g$, respectively). It can be seen clearly that with respect to all error measures, Sinkhorn converges linearly while Sinkhorn--Newton converges roughly quadratically, as expected, with Sinkhorn--Newton significantly outperforming classical Sinkhorn for this choice of parameters. The same behavior holds if the error is measured as a function of runtime; see \cref{fig:errors:time}.

\subsection{Dependence on the regularization strength}
\label{sec:dependence_regularization_strength}

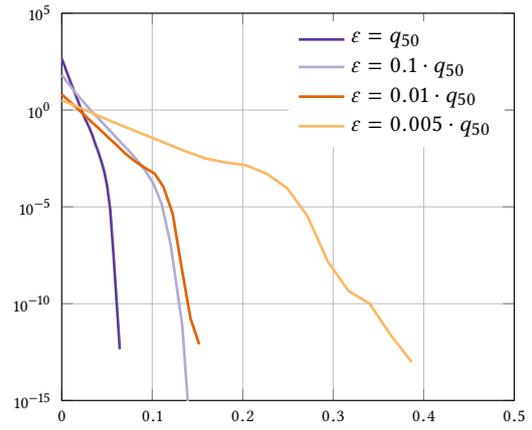
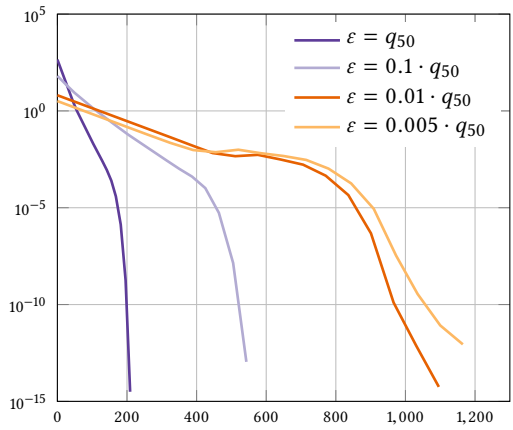
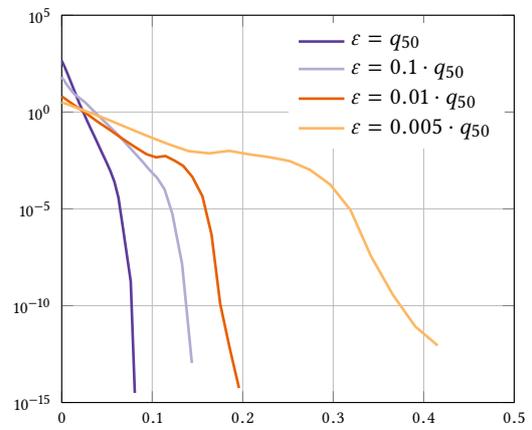
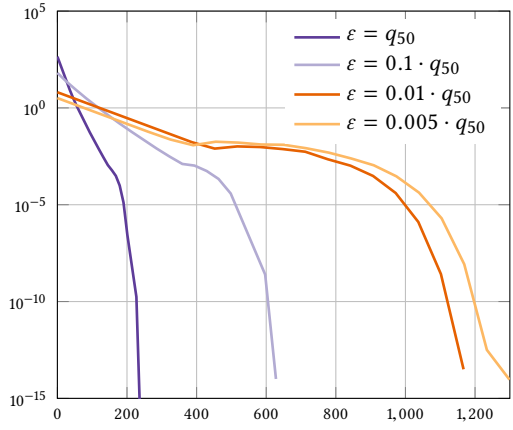
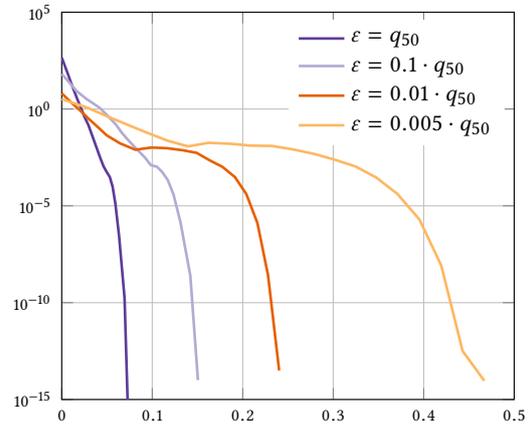
\begin{figure*}
    \begin{subfigure}{0.45\linewidth}
        \centering
        \begin{tikzpicture}

\begin{axis}[%
width=0.9\linewidth,
scale only axis,
xmin=0,
xmax=1300,
ymode=log,
ymin=1e-15,
ymax=100000,
xmajorgrids,
ymajorgrids,
]

\addplot [color=PuOr-C, line width=1.0]
  table[row sep=crcr]{%
0	446.812713703415\\
16	60.4691658760494\\
52	1.10712360119674\\
62	0.407025604817783\\
98	0.0198755617484343\\
111	0.00706229860217164\\
124	0.00236813165663525\\
137	0.00068128721986078\\
150	0.000128083251513789\\
165	7.50475365232052e-06\\
180	2.93545871284533e-08\\
212	4.45500407397775e-13\\
};
\addlegendentry{$\epsilon=q_{50}$}

\addplot [color=PuOr-F, line width=1.0]
  table[row sep=crcr]{%
0	64.3998251532039\\
23	23.6911090335263\\
47	8.71520930624603\\
72	3.20588370980916\\
101	1.17911615954804\\
131	0.433510239570681\\
163	0.159217675342638\\
229	0.0211952963272121\\
265	0.00755041411508417\\
301	0.00255566906037838\\
338	0.000763177151154049\\
375	0.0001642101913696\\
415	1.37329645205837e-05\\
457	1.10296283532223e-07\\
500	6.86589088534585e-12\\
540	6.93889390390763e-18\\
};
\addlegendentry{$\epsilon=0.1\cdot q_{50}$}

\addplot [color=PuOr-M, line width=1.0]
  table[row sep=crcr]{%
0	6.43927096164695\\
390	0.015619479542498\\
453	0.00564346945786581\\
511	0.00216214591580734\\
576	0.0010428575664782\\
641	0.000524864356284161\\
706	0.000107486178602209\\
771	4.27832503711571e-06\\
836	8.89349862265446e-09\\
901	1.73537269623763e-11\\
965	8.00775244724732e-13\\
};
\addlegendentry{$\epsilon=0.01\cdot q_{50}$}

\addplot [color=PuOr-J, line width=1.0]
  table[row sep=crcr]{%
0	3.21965901909914\\
390	0.00786206430263732\\
455	0.00323988921060775\\
520	0.00199048019734671\\
585	0.00141668360609476\\
650	0.000514392395974412\\
715	9.42268195574766e-05\\
780	3.37950450264056e-06\\
845	1.45093637261147e-08\\
910	4.43080668075645e-10\\
974	1.00074010276466e-10\\
1039	2.37227490255947e-12\\
1102	9.65889694615072e-14\\
};
\addlegendentry{$\epsilon=0.005\cdot q_{50}$}

\end{axis}
\end{tikzpicture}%
        \caption{$\gamma=0.5$ (CG iterations)}
        \label{fig:mnist:05:its}
    \end{subfigure}
    \hspace{1cm}
    \begin{subfigure}{0.45\linewidth}
        \centering
        \begin{tikzpicture}

\begin{axis}[%
width=0.9\linewidth,
scale only axis,
xmin=0,
xmax=0.5,
ymode=log,
ymin=1e-15,
ymax=100000,
xmajorgrids,
ymajorgrids,
]
\addplot [color=PuOr-C, line width=1.0]
  table[row sep=crcr]{%
3.50000000004513e-05	446.812713703421\\
0.00358300000000078	164.372948765068\\
0.00700900000000004	60.4691658760507\\
0.0141960000000001	8.1832524166697\\
0.017887	3.01018769940181\\
0.0213479999999997	1.10712360119677\\
0.0252090000000003	0.407025604817795\\
0.0295939999999995	0.149474380605369\\
0.0335900000000002	0.0547277561820941\\
0.0366689999999998	0.0198755617484341\\
0.0403140000000004	0.00706229860217152\\
0.0435870000000005	0.00236813165663521\\
0.0468910000000005	0.000681287219860772\\
0.0500930000000004	0.000128083251513786\\
0.0535359999999994	7.50475365232064e-06\\
0.0572900000000001	2.9354587128454e-08\\
0.0641130000000008	4.4550040739777e-13\\
};
\addlegendentry{$\epsilon=q_{50}$}

\addplot [color=PuOr-F, line width=1.0]
  table[row sep=crcr]{%
4.09999999995136e-05	64.3998251532097\\
0.0143109999999993	8.71520930624646\\
0.0224179999999983	3.2058837098094\\
0.0304089999999988	1.17911615954814\\
0.0653910000000018	0.0211952963272138\\
0.0744590000000009	0.00755041411508336\\
0.0833739999999992	0.00255566906037844\\
0.0921400000000006	0.000763177151154143\\
0.101210999999999	0.000164210191369591\\
0.110492000000001	1.37329645205838e-05\\
0.120417	1.10296283532211e-07\\
0.133293999999999	6.86589088534636e-12\\
0.142499999999998	6.93889390390723e-18\\
};
\addlegendentry{$\epsilon=0.1\cdot q_{50}$}

\addplot [color=PuOr-M, line width=1.0]
  table[row sep=crcr]{%
3.50000000004513e-05	6.43927096164771\\
0.0102969999999996	2.36861317710938\\
0.0204199999999997	0.871102612999858\\
0.0302089999999993	0.320201291451154\\
0.0404459999999993	0.117541514715797\\
0.0503020000000003	0.0430019406141536\\
0.0605049999999991	0.0156194795424987\\
0.0702569999999998	0.00564346945786528\\
0.0815359999999998	0.00216214591580726\\
0.0918960000000002	0.00104285756647828\\
0.102637	0.000524864356284197\\
0.112401999999999	0.000107486178602214\\
0.122661000000001	4.2783250371152e-06\\
0.132159	8.89349862265472e-09\\
0.142314000000001	1.73537269623769e-11\\
0.151939	8.00775244724771e-13\\
};
\addlegendentry{$\epsilon=0.01\cdot q_{50}$}

\addplot [color=PuOr-J, line width=1.0]
  table[row sep=crcr]{%
3.00000000006406e-05	3.21965901909927\\
0.0222960000000008	1.18418579073643\\
0.0437759999999994	0.435379904238424\\
0.0660939999999997	0.15991789594693\\
0.0893420000000003	0.0586037900705452\\
0.112538000000001	0.0213931271383226\\
0.135370999999999	0.0078620643026368\\
0.157971	0.00323988921060731\\
0.180308999999999	0.00199048019734721\\
0.203382	0.00141668360609454\\
0.226521	0.000514392395974461\\
0.248927999999999	9.42268195574827e-05\\
0.271417	3.37950450264015e-06\\
0.294231999999999	1.45093637261146e-08\\
0.317123	4.43080668075663e-10\\
0.340249999999999	1.00074010276479e-10\\
0.363573000000001	2.37227490255953e-12\\
0.386865	9.65889694615195e-14\\
};
\addlegendentry{$\epsilon=0.005\cdot q_{50}$}

\end{axis}
\end{tikzpicture}%
        \caption{$\gamma=0.5$ (run time in seconds)}
        \label{fig:mnist:05:time}
    \end{subfigure}

    \begin{subfigure}{0.45\linewidth}
        \centering
        \begin{tikzpicture}

\begin{axis}[%
width=0.9\linewidth,
scale only axis,
xmin=0,
xmax=1300,
ymode=log,
ymin=1e-15,
ymax=100000,
xmajorgrids,
ymajorgrids,
]

\addplot [color=PuOr-C, line width=1.0]
  table[row sep=crcr]{%
0	446.813179762451\\
45	3.01038442398048\\
55	1.10731929655964\\
103	0.0200803229169965\\
129	0.00260044528323108\\
142	0.000889971248782125\\
155	0.000249162327593543\\
168	3.86313398615028e-05\\
182	1.3571226096899e-06\\
196	1.79930626474641e-09\\
209	3.13475374130729e-15\\
};
\addlegendentry{$\epsilon=q_{50}$}

\addplot [color=PuOr-F, line width=1.0]
  table[row sep=crcr]{%
0	64.4002912122598\\
24	23.6914036520228\\
49	8.71544087770803\\
107	1.17931625280855\\
171	0.159415367587797\\
204	0.0585133503437811\\
313	0.00281862225942695\\
350	0.00102528724643555\\
387	0.000404073711722002\\
425	0.000103177609590683\\
464	5.44377841706495e-06\\
505	1.39525737768999e-08\\
543	1.10021692277519e-13\\
};
\addlegendentry{$\epsilon=0.1\cdot q_{50}$}

\addplot [color=PuOr-M, line width=1.0]
  table[row sep=crcr]{%
0	6.43973702069222\\
325	0.0433323885106849\\
388	0.016165687940146\\
446	0.00666367711086923\\
511	0.00458959606153572\\
576	0.0053836226409746\\
641	0.003103981152198\\
706	0.00167387624020947\\
771	0.000450609606132024\\
836	4.50148379913948e-05\\
901	4.7493336484936e-07\\
966	1.26322676277666e-10\\
1031	7.46862641176578e-13\\
1096	5.61183044478558e-15\\
};
\addlegendentry{$\epsilon=0.01\cdot q_{50}$}

\addplot [color=PuOr-J, line width=1.0]
  table[row sep=crcr]{%
0	3.22012507814346\\
260	0.0590187569072361\\
325	0.0221677958132257\\
390	0.00953982728742024\\
455	0.0073679265821345\\
520	0.0099043740603314\\
585	0.00651592699321286\\
650	0.00468329224144289\\
715	0.00294052015581286\\
779	0.00104699130969165\\
844	0.000180481925881224\\
909	8.80266693363834e-06\\
973	3.41490780887254e-08\\
1035	3.50427197676262e-10\\
1100	8.3694292815615e-12\\
1165	8.8110768681189e-13\\
};
\addlegendentry{$\epsilon=0.005\cdot q_{50}$}

\end{axis}
\end{tikzpicture}%
        \caption{$\gamma=0.1$ (CG iterations)}
        \label{fig:mnist:01:its}
    \end{subfigure}
    \hspace{1cm}
    \begin{subfigure}{0.45\linewidth}
        \centering
        \begin{tikzpicture}

\begin{axis}[%
width=0.9\linewidth,
scale only axis,
xmin=0,
xmax=0.5,
ymode=log,
ymin=1e-15,
ymax=100000,
xmajorgrids,
ymajorgrids,
]

\addplot [color=PuOr-C, line width=1.0]
  table[row sep=crcr]{%
4.09999999995136e-05	446.813179762464\\
0.00427999999999962	164.373243371778\\
0.00809500000000085	60.4693974109582\\
0.0155069999999995	8.18345223509655\\
0.0191759999999999	3.01038442398053\\
0.0274439999999991	0.407221216399037\\
0.0317070000000008	0.149670753232163\\
0.0361609999999999	0.0549264910877305\\
0.0448280000000008	0.00727990024520808\\
0.0493439999999996	0.0026004452832311\\
0.0537949999999991	0.000889971248782122\\
0.0582170000000009	0.000249162327593542\\
0.0626239999999996	3.86313398615018e-05\\
0.0672779999999999	1.35712260968988e-06\\
0.0762330000000002	1.79930626474645e-09\\
0.0807310000000001	3.13475374130734e-15\\
};
\addlegendentry{$\epsilon=q_{50}$}

\addplot [color=PuOr-F, line width=1.0]
  table[row sep=crcr]{%
3.59999999997029e-05	64.4002912122541\\
0.00614999999999988	23.691403652024\\
0.0135319999999997	8.71544087770699\\
0.0256989999999995	3.20609214835025\\
0.0348980000000001	1.17931625280859\\
0.0539919999999992	0.159415367587791\\
0.0634099999999993	0.0585133503437881\\
0.0804430000000007	0.00778418848375276\\
0.0890360000000001	0.00281862225942724\\
0.0970019999999998	0.00102528724643568\\
0.105319	0.00040407371172198\\
0.113605	0.000103177609590689\\
0.122273	5.44377841706511e-06\\
0.13302	1.39525737768986e-08\\
0.143818	1.10021692277529e-13\\
};
\addlegendentry{$\epsilon=0.1\cdot q_{50}$}

\addplot [color=PuOr-M, line width=1.0]
  table[row sep=crcr]{%
3.59999999997029e-05	6.43973702069209\\
0.0131870000000003	2.3689099383576\\
0.0266889999999993	0.871340793131513\\
0.0395129999999995	0.320427958003056\\
0.0534029999999994	0.11779107021747\\
0.0666460000000004	0.0433323885106838\\
0.0804650000000002	0.0161656879401443\\
0.0936730000000008	0.00666367711087083\\
0.104323000000001	0.00458959606153551\\
0.114559	0.00538362264097345\\
0.124620999999999	0.00310398115219814\\
0.134168000000001	0.00167387624020945\\
0.144337999999999	0.00045060960613205\\
0.155465	4.50148379913937e-05\\
0.165526	4.74933364849259e-07\\
0.175095000000001	1.26322676277657e-10\\
0.185313000000001	7.46862641176626e-13\\
0.195887000000001	5.61183044478496e-15\\
};
\addlegendentry{$\epsilon=0.01\cdot q_{50}$}

\addplot [color=PuOr-J, line width=1.0]
  table[row sep=crcr]{%
4.20000000005416e-05	3.22012507814364\\
0.0249880000000005	1.18449007150544\\
0.0494389999999996	0.435640243669278\\
0.0717110000000005	0.160205649903543\\
0.094322	0.0590187569072491\\
0.117671	0.0221677958132202\\
0.140421	0.00953982728741958\\
0.162767000000001	0.00736792658213309\\
0.185226	0.00990437406033128\\
0.207333	0.0065159269932131\\
0.229573	0.00468329224144304\\
0.251856999999999	0.00294052015581351\\
0.274194	0.00104699130969167\\
0.296661	0.000180481925881264\\
0.319056	8.80266693363726e-06\\
0.341950000000001	3.41490780887341e-08\\
0.366196	3.50427197676196e-10\\
0.390731000000001	8.36942928156281e-12\\
0.415374	8.81107686812045e-13\\
};
\addlegendentry{$\epsilon=0.005\cdot q_{50}$}

\end{axis}
\end{tikzpicture}%
        \caption{$\gamma=0.1$ (run time in seconds)}
        \label{fig:mnist:01:time}
    \end{subfigure}

    \begin{subfigure}{0.45\linewidth}
        \centering
        \begin{tikzpicture}

\begin{axis}[%
width=0.9\linewidth,
scale only axis,
xmin=0,
xmax=1300,
ymode=log,
ymin=1e-15,
ymax=100000,
xmajorgrids,
ymajorgrids,
]

\addplot [color=PuOr-C, line width=1.0]
  table[row sep=crcr]{%
0	446.813623582521\\
36	8.1836425469773\\
46	3.01057183812063\\
57	1.10750586206243\\
93	0.0551197257411801\\
119	0.00750799936551102\\
132	0.00284977627628515\\
145	0.00111216085073352\\
157	0.000590686783370153\\
168	0.000308606587096374\\
179	9.80539936490465e-05\\
190	1.29305207398175e-05\\
201	2.79143335012869e-07\\
227	1.74250707613005e-10\\
237	2.82759926584222e-16\\
};
\addlegendentry{$\epsilon=q_{50}$}

\addplot [color=PuOr-F, line width=1.0]
  table[row sep=crcr]{%
0	64.4007350323011\\
25	23.6916842225179\\
53	8.71566143106247\\
82	3.2062907285581\\
113	1.17950703647268\\
145	0.433896396042147\\
178	0.159605330734306\\
213	0.058708959997985\\
285	0.0080285130238959\\
322	0.00309817138973513\\
359	0.00127744050767311\\
394	0.00104948011013136\\
429	0.000551235898885522\\
463	0.000215291545174612\\
498	3.79422723384414e-05\\
597	2.4803055772433e-09\\
628	1.00587940754526e-14\\
};
\addlegendentry{$\epsilon=0.1\cdot q_{50}$}

\addplot [color=PuOr-M, line width=1.0]
  table[row sep=crcr]{%
0	6.44018084073573\\
387	0.0168783825012995\\
452	0.00791346358649756\\
517	0.0102663301759433\\
582	0.0095542315145237\\
647	0.00748703399462987\\
712	0.00546233394742603\\
777	0.00222835339266413\\
842	0.00104134647099997\\
907	0.000305913906152141\\
972	4.01757335889449e-05\\
1037	1.28771523774375e-06\\
1102	2.60940067210028e-09\\
1167	3.1752378504272e-14\\
};
\addlegendentry{$\epsilon=0.01\cdot q_{50}$}

\addplot [color=PuOr-J, line width=1.0]
  table[row sep=crcr]{%
0	3.22056889818561\\
260	0.0595444731510985\\
325	0.0232504736147584\\
390	0.0118253941105835\\
455	0.0179425681182147\\
520	0.0161222134242907\\
585	0.0128936069294315\\
650	0.0124797264988977\\
715	0.00825895977519991\\
780	0.0048897420366676\\
844	0.00245147826364658\\
909	0.00108985801714871\\
974	0.00029285409615693\\
1039	4.21522615045158e-05\\
1104	1.98450556654309e-06\\
1169	8.41118503232998e-09\\
1234	3.18006038169189e-13\\
1299	9.25301502086208e-15\\
};
\addlegendentry{$\epsilon=0.005\cdot q_{50}$}

\end{axis}
\end{tikzpicture}%
        \caption{$\gamma=0.01$ (CG iterations)}
        \label{fig:mnist:001:its}
    \end{subfigure}
    \hspace{1cm}
    \begin{subfigure}{0.45\linewidth}
        \centering
        \begin{tikzpicture}

\begin{axis}[%
width=0.9\linewidth,
scale only axis,
xmin=0,
xmax=0.5,
ymode=log,
ymin=1e-15,
ymax=100000,
xmajorgrids,
ymajorgrids,
]

\addplot [color=PuOr-C, line width=1.0]
  table[row sep=crcr]{%
3.79999999999825e-05	446.813623582506\\
0.00702399999999948	60.4696179013948\\
0.0103010000000001	22.2455070382884\\
0.0174730000000007	3.01057183812068\\
0.0213359999999998	1.10750586206243\\
0.0249989999999993	0.407408057354546\\
0.0294340000000002	0.149859268981801\\
0.0427149999999994	0.00284977627628507\\
0.0461270000000003	0.00111216085073349\\
0.0493710000000007	0.000590686783370155\\
0.0531079999999999	0.000308606587096369\\
0.0561959999999999	9.80539936490473e-05\\
0.0592649999999999	1.29305207398171e-05\\
0.0633499999999998	2.7914333501286e-07\\
0.0694090000000003	1.74250707613011e-10\\
0.073086	2.8275992658422e-16\\
};
\addlegendentry{$\epsilon=q_{50}$}

\addplot [color=PuOr-F, line width=1.0]
  table[row sep=crcr]{%
3.30000000001718e-05	64.4007350322966\\
0.00829399999999936	23.6916842225158\\
0.0161669999999994	8.71566143106183\\
0.0277550000000009	3.20629072855823\\
0.0416740000000004	1.1795070364728\\
0.0510920000000006	0.433896396042099\\
0.0597580000000004	0.159605330734294\\
0.0664879999999997	0.0587089599979859\\
0.0745819999999995	0.021617433621713\\
0.0833290000000009	0.00802851302389603\\
0.0920199999999998	0.00309817138973497\\
0.0983920000000005	0.00127744050767321\\
0.105193999999999	0.00104948011013146\\
0.111115	0.000551235898885543\\
0.117179	0.000215291545174584\\
0.123559	3.79422723384444e-05\\
0.131244000000001	1.3866826956959e-06\\
0.142111	2.48030557724321e-09\\
0.150518	1.00587940754515e-14\\
};
\addlegendentry{$\epsilon=0.1\cdot q_{50}$}

\addplot [color=PuOr-M, line width=1.0]
  table[row sep=crcr]{%
3.30000000001718e-05	6.4401808407347\\
0.00961900000000071	2.36919385027132\\
0.0200220000000009	0.871571650796927\\
0.0295480000000001	0.320654941944022\\
0.0397949999999998	0.11805877930159\\
0.0497359999999993	0.0437265293450478\\
0.0644419999999997	0.0168783825012962\\
0.0814000000000004	0.00791346358649711\\
0.0990520000000004	0.0102663301759435\\
0.116064	0.00955423151452316\\
0.133490999999999	0.00748703399462873\\
0.149095000000001	0.00546233394742675\\
0.16361	0.00222835339266417\\
0.177521	0.00104134647099987\\
0.191423	0.000305913906152085\\
0.204027999999999	4.0175733588936e-05\\
0.216272	1.28771523774402e-06\\
0.227966	2.6094006721003e-09\\
0.240247	3.17523785042794e-14\\
};
\addlegendentry{$\epsilon=0.01\cdot q_{50}$}

\addplot [color=PuOr-J, line width=1.0]
  table[row sep=crcr]{%
4.09999999995136e-05	3.22056889818624\\
0.0282260000000001	1.18478529714878\\
0.0494839999999996	0.435905739622146\\
0.0717859999999995	0.160527844087322\\
0.0944660000000006	0.0595444731511024\\
0.116908	0.0232504736147539\\
0.139663000000001	0.0118253941105824\\
0.162110999999999	0.017942568118215\\
0.186705	0.0161222134242874\\
0.208995	0.0128936069294299\\
0.232155000000001	0.0124797264988988\\
0.255412	0.00825895977520113\\
0.278421	0.00488974203666706\\
0.301339	0.00245147826364609\\
0.32452	0.00108985801714896\\
0.347458	0.000292854096156952\\
0.370729000000001	4.21522615045109e-05\\
0.395251	1.98450556654263e-06\\
0.419366	8.41118503232956e-09\\
0.443004999999999	3.18006038169116e-13\\
0.466989	9.25301502086031e-15\\
};
\addlegendentry{$\epsilon=0.005\cdot q_{50}$}

\end{axis}
\end{tikzpicture}%
        \caption{$\gamma=0.01$ (run time in seconds)}
        \label{fig:mnist:001:time}
    \end{subfigure}
    \caption{Constraint violation for different offsets $\gamma$ and different regularization parameters $\epsilon$.}
    \label{fig:mnist}
\end{figure*}

The second example addresses the dependence of the Sinkhorn--Newton method on the problem parameters. In particular, we consider the dependence on $\epsilon$ and on the minimal value of $\a$ and $\b$ (via the corresponding transport plans $\P$), since these enter into the convergence rate estimate \eqref{eq:newton_constant}. Here, we take an example which is also used in \textcite{Cuturi2013}: computing the transport distances between different images from the MNIST database, which contains $28\times 28$ images of handwritten digits. We consider these as discrete distributions of dimension $28^{2} = 784$ on $[0, 1]^{2}$, to which we add a small offset $\gamma$ before normalizing to unit mass as before. Here, the tolerance for both the Newton and the CG iteration is set to $10^{-12}$, and the maximum number of CG iterations is fixed at $66$. The entropic regularization parameter $\epsilon$ is chosen as multiples of the median of the cost (which is $q_{50} = 0.2821$ in this case).

\Cref{fig:mnist} shows the convergence history for different offsets $\gamma\in\{0.5, 0.1, 0.01\}$ and $\epsilon\in q_{50}\cdot\{1, 0.1,  0.01, 0.005\}$, where we again report the constraint violation both as a function of CG iterations and of the run time in seconds. Comparing \crefrange{fig:mnist:05:its}{fig:mnist:001:time}, we see that as $\epsilon$ decreases, an increasing number of CG iterations is required to achieve the prescribed tolerance. However, the convergence seems to be robust in $\epsilon$ at least for larger values of $\epsilon$ and only moderately deteriorate for $\epsilon \leq 0.01q_{50}$.

\subsection{Dependence on the problem dimension}
\label{sec:dependence_problem_dimension}

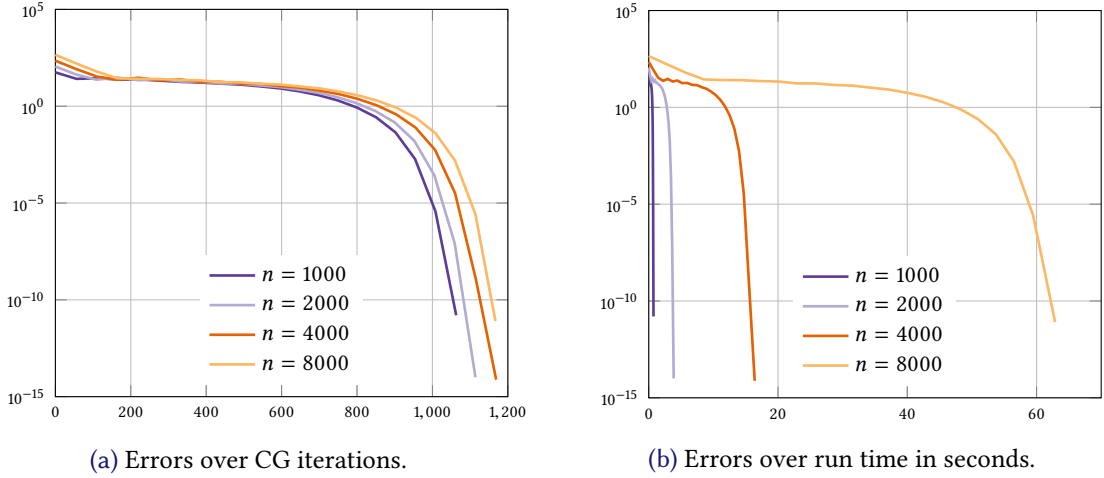
\begin{figure*}
    \begin{subfigure}{0.45\linewidth}
        \centering
        \begin{tikzpicture}

\begin{axis}[%
width=0.9\linewidth,
scale only axis,
xmin=0,
xmax=1200,
ymode=log,
ymin=1e-15,
ymax=100000,
xmajorgrids,
ymajorgrids,
legend style={at={(0.5,0.03)},anchor=south}
]
\addplot [color=PuOr-C, line width=1.0]
  table[row sep=crcr]{%
0	55.9587841728713\\
56	25.8381811258524\\
112	27.8733942008674\\
168	23.6413130411452\\
223	24.0897421379564\\
332	18.072302590235\\
386	16.6540811221653\\
438	14.5856126912665\\
491	13.014516263369\\
544	10.5439398807965\\
596	8.29535913718132\\
647	5.84418788173862\\
698	3.7110856808646\\
749	1.98418403845342\\
800	0.851707644208309\\
851	0.264502243434872\\
902	0.0445454614856351\\
954	0.0018584373150116\\
1008	3.70532734605951e-06\\
1063	1.59520174847177e-11\\
};
\addlegendentry{$n=1000$}

\addplot [color=PuOr-F, line width=1.0]
  table[row sep=crcr]{%
0	112.008682191131\\
55	43.3689598891171\\
109	23.4510129022239\\
165	30.377442628086\\
221	21.80704738245\\
276	25.0459836318234\\
330	19.0082824457469\\
383	18.2100526528908\\
436	15.2230845735031\\
489	14.4079851631448\\
594	9.7814494406789\\
645	7.22016730624712\\
696	4.96624263565133\\
747	2.94263769456753\\
798	1.45653164195543\\
849	0.560078202784891\\
900	0.145161199200274\\
952	0.0163383900311555\\
1005	0.000272998107264497\\
1059	8.33115195719806e-08\\
1114	1.01030295240889e-14\\
};
\addlegendentry{$n=2000$}

\addplot [color=PuOr-M, line width=1.0]
  table[row sep=crcr]{%
0	224.108492909422\\
55	83.3476706177233\\
110	33.9205307196888\\
166	23.297275816247\\
221	29.4400234501647\\
277	21.5038405133707\\
331	24.1728131777303\\
385	17.8111180640235\\
439	18.1909741707009\\
491	14.5443852623375\\
544	14.0547284688977\\
597	10.9710020533087\\
649	9.07365811298037\\
700	6.42478057543656\\
751	4.26668152375795\\
802	2.38382182227042\\
853	1.09550574501057\\
904	0.377746706814715\\
955	0.0794455223324442\\
1007	0.00553088737414716\\
1060	3.24367791553139e-05\\
1115	1.20482790411075e-09\\
1169	7.5495165674529e-15\\
};
\addlegendentry{$n=4000$}

\addplot [color=PuOr-J, line width=1.0]
  table[row sep=crcr]{%
0	448.308128051643\\
54	165.261717980084\\
109	62.2727177033941\\
164	27.523401639623\\
220	25.7395160310653\\
276	25.2922788389645\\
331	22.6853293750808\\
385	21.5054511617477\\
439	17.2263307463899\\
492	17.1736442508011\\
545	14.2808202259333\\
598	13.1440130825014\\
702	8.19471506931379\\
753	5.67343044092523\\
804	3.58128386078495\\
855	1.89023090223913\\
906	0.803201765077351\\
957	0.245966324463997\\
1008	0.0401310241405815\\
1060	0.00154897429071451\\
1114	2.62012637608518e-06\\
1167	8.14492917555754e-12\\
};
\addlegendentry{$n=8000$}

\end{axis}
\end{tikzpicture}%
        \caption{Errors over CG iterations.}
        \label{fig:discr:its}
    \end{subfigure}
    \hspace{1cm}
    \begin{subfigure}{0.45\linewidth}
        \centering
        \begin{tikzpicture}

\begin{axis}[%
width=0.9\linewidth,
scale only axis,
xmin=0,
xmax=70,
ymode=log,
ymin=1e-15,
ymax=100000,
xmajorgrids,
ymajorgrids,
legend style={at={(0.5,0.03)},anchor=south}
]
\addplot [color=PuOr-C, line width=1.0]
  table[row sep=crcr]{%
0.0994650000000004	55.9587841728642\\
0.167619	25.8381811258571\\
0.19608	27.8733942008715\\
0.223445	23.6413130411463\\
0.250576000000001	24.0897421379539\\
0.286593	20.9326317265704\\
0.312302000000001	18.0723025902389\\
0.337975	16.6540811221656\\
0.421109	10.5439398807966\\
0.454283999999999	8.29535913717959\\
0.486554	5.84418788173755\\
0.518494	3.71108568086399\\
0.568514	0.851707644208215\\
0.595421	0.264502243434928\\
0.621247	0.0445454614856433\\
0.647862	0.00185843731501189\\
0.683007999999999	3.70532734605966e-06\\
0.733174999999999	1.59520174847216e-11\\
};
\addlegendentry{$n=1000$}

\addplot [color=PuOr-F, line width=1.0]
  table[row sep=crcr]{%
0.000527999999999196	112.00868219113\\
0.205937	43.3689598891283\\
0.397618	23.4510129022205\\
0.581898000000001	30.3774426280885\\
0.767073	21.8070473824492\\
0.948765	25.0459836318286\\
1.130799	19.0082824457437\\
1.310923	18.2100526528879\\
1.492301	15.2230845735034\\
1.673394	14.4079851631427\\
2.032019	9.78144944067666\\
2.206815	7.22016730624798\\
2.380053	4.96624263565147\\
2.555748	2.94263769456684\\
2.729581	1.45653164195513\\
2.904222	0.560078202784914\\
3.078116	0.145161199200264\\
3.255555	0.0163383900311544\\
3.437053	0.000272998107264444\\
3.622724	8.33115195719714e-08\\
3.859601	1.01030295240889e-14\\
};
\addlegendentry{$n=2000$}

\addplot [color=PuOr-M, line width=1.0]
  table[row sep=crcr]{%
0.000105999999998829	224.108492909466\\
0.733015999999999	83.347670617711\\
1.466399	33.9205307196827\\
2.21231	23.2972758162491\\
2.945893	29.4400234501649\\
3.69269	21.5038405133686\\
4.420255	24.1728131777262\\
5.173928	17.8111180640205\\
5.957742	18.1909741706991\\
6.693902	14.5443852623355\\
7.454275	14.0547284689002\\
8.21413	10.9710020533103\\
8.953922	9.07365811298222\\
9.672459	6.42478057543692\\
10.388936	4.26668152375791\\
11.104599	2.38382182227005\\
11.816696	1.09550574501047\\
12.529631	0.377746706814752\\
13.244402	0.0794455223324631\\
13.968208	0.00553088737414843\\
14.710307	3.24367791553203e-05\\
15.475092	1.20482790411102e-09\\
16.39269	7.54951656745105e-15\\
};
\addlegendentry{$n=4000$}

\addplot [color=PuOr-J, line width=1.0]
  table[row sep=crcr]{%
0.000115000000000975	448.308128051592\\
2.824607	165.261717980062\\
5.694746	62.272717703387\\
8.556808	27.5234016396166\\
11.454647	25.7395160310687\\
14.368761	25.2922788389703\\
17.220995	22.6853293750767\\
20.100649	21.5054511617441\\
22.988253	17.2263307463868\\
25.86129	17.1736442507983\\
28.69906	14.2808202259343\\
31.591769	13.144013082499\\
37.215354	8.19471506931177\\
39.940635	5.67343044092542\\
42.728688	3.58128386078425\\
45.475313	1.89023090223895\\
48.220894	0.803201765077443\\
50.950886	0.24596632446396\\
53.723289	0.0401310241405907\\
56.517732	0.00154897429071443\\
59.470611	2.620126376085e-06\\
62.870142	8.14492917555754e-12\\
};
\addlegendentry{$n=8000$}

\end{axis}
\end{tikzpicture}%
        \caption{Errors over run time in seconds.}
        \label{fig:discr:time}
    \end{subfigure}
    \caption{Constraint violation for different mesh sizes $\n$}
    \label{fig:discr}
\end{figure*}

We finally address the dependence on the dimension of the problem. For this purpose, we discretize the unit interval $[0, 1]$ using $\n$ equidistant points $\x_{i}\in [0, 1]$ and take
\begin{subequations}
    \begin{align}
        \tilde \a_{i} &= \e^{-100(\x_{i} - 0.2)^{2}} + \e^{-20\abs{\x_{i} - 0.4}} + 10^{-2},\\
        \tilde \b_{j} &= \e^{-100(\x_{j} - 0.6)^{2}} + 10^{-2},
    \end{align}%
\end{subequations}
which are again normalized to unit mass to obtain $\a$ and $\b$. The regularization parameter is fixed at $\epsilon = 10^{-3}$. Moreover, the inner and outer tolerances are here set to $10^{-10}$, and the maximum number of CG iterations is coupled to the mesh size via $\lceil \n / 12\rceil$.

\Cref{fig:discr} shows the convergence behavior of Sinkhorn--Newton for $\n\in\{1000, 2000, 4000, 8000\}$. As can be seen from \cref{fig:discr:its}, the behavior is nearly independent of $\n$; in particular, the number of CG iterations required to reach the prescribed tolerance stays almost the same. (This is also true for the Newton method itself with $21, 22, 23$ and $23$ iterations.) Since each CG iteration involves two dense matrix--vector products with complexity ${\cal O}(\n^{2})$, the total run time scales quadratically; see \cref{fig:discr:time}.

\section{Proof of Theorem \labelcref{thm:newton_constant}}
\label{sec:proof}

For the sake of presentation, we restrict ourselves to the case $\m = \n$ here. However, in the end, we suggest how the proof can be generalized to the case $\m \neq \n$.

To estimate~\eqref{eq:newton_condition} and in particular $\J_{F}(\y^{k}) - \J_{F}(\etaa)$ for $\y^{k} = (\f^{k}, \g^{k})$ and $\etaa = (\alphaa,\betaa)\in\R^{\n}\times\R^{\n}$ we observe that
\begin{equation*}
    \J_{F}(\y^{k})-\J_{F}(\etaa) = \frac1\epsilon\left[\e^{\frac{-c_{ij}-\f^{k}_{i}-\g^{k}_{j}}\epsilon} - \e^{\frac{-c_{ij}-\alphaa_{i}-\betaa_{j}}\epsilon}\right]_{ij}
    = \left[\P_{ij}^{k}(1 - \e^{\frac{\f^{k}_{i}-\alphaa_{i} + \g^{k}_{j}-\betaa_{j}}\epsilon})\right]_{ij}.
\end{equation*}
To keep the notation concise, we abbreviate $\psii = (\phii, \gammaa) \coloneqq \y^{k} - \etaa$ and also write $\y$ and $\P$ for $\y^{k}$ and $\P^{k}$, respectively. Then, we compute
\begin{equation*}
    \begin{aligned}
        \J_{F}(\y)^{-1}[\J_{F}(\y) - \J_{F}(\etaa)](\y - \etaa)
        &=\ \J_{F}(\y)^{-1}
        \begin{bmatrix}
            \left(\sum_{j}\P_{ij}(\e^{(\phii_{i} + \gammaa_{j})/\epsilon}-1)(\phii_{i} + \gammaa_{j})/\epsilon\right)_{i}\\
            \left(\sum_{i}\P_{ij}(\e^{(\phii_{i} + \gammaa_{j})/\epsilon}-1)(\phii_{i} + \gammaa_{j})/\epsilon\right)_{j}    
        \end{bmatrix}\\
        &=\ \J_{F}(\y)^{-1}
        \begin{bmatrix}
            \left(\sum\limits_{j}\P_{ij}\sum\limits_{k=2}^\infty\frac1{(k-1)!\epsilon^k}\sum\limits_{l=0}^k{k\choose l}\phii_{i}^l\gammaa_{j}^{k-l}\right)_{i}\\
            \left(\sum\limits_{i}\P_{ij}\sum\limits_{k=2}^\infty\frac1{(k-1)!\epsilon^k}\sum\limits_{l=0}^k{k\choose l}\phii_{i}^l\gammaa_{j}^{k-l}\right)_{j}    
        \end{bmatrix}\\
        &=\ \sum\limits_{k=2}^{\infty}\sum\limits_{l=0}^{k}\binom{k}{l} \frac{1}{(k-1)!\epsilon^{k}} \J_{F}(\y)^{-1}
        \begin{bmatrix}
            \left(\sum_{j}\P_{ij}\phii_{i}^l\gammaa_{j}^{k-l}\right)_{i}\\
            \left(\sum_{i}\P_{ij}\phii_{i}^l\gammaa_{j}^{k-l}\right)_{j}    
        \end{bmatrix}
    \end{aligned}
\end{equation*}
where all exponents are applied componentwise.
Now we first treat only the summands for $l=0$ and $l=k$.
For those terms \eqref{eq:jacobian} immediately implies
\begin{equation*}
    \sum\limits_{k=2}^{\infty}\sum\limits_{l=0,k}\binom{k}{l}
    \frac{1}{(k-1)!\epsilon^{k}}
    \J_{F}(\y)^{-1}
    \begin{bmatrix}
        \left(\sum_{j}\P_{ij}(\phii_{i}^{k} + \gammaa_{j}^{k})\right)_{i}\\
        \left(\sum_{i}\P_{ij}(\phii_{i}^{k} + \gammaa_{j}^{k})\right)_{j}    
    \end{bmatrix}\label{eq:deuflhard-est-l=0-k}
    =\sum\limits_{k=2}^{\infty}\sum\limits_{l=0,k}
    \frac{1}{(k-1)!\epsilon^{k-1}}
    \begin{bmatrix}
        \phii^{k}\\\gammaa^{k}
    \end{bmatrix}
    =\begin{bmatrix}
        (\e^{\phii/\epsilon}-1)\phii\\(\e^{\gammaa/\epsilon}-1)\gammaa
    \end{bmatrix},
\end{equation*}
which has supremum norm bounded by $(\e^{\frac1\epsilon}-1)\norm{(\phii,\gammaa)}_{\infty}^2$ for all $\norm{(\phii,\gammaa)}_{\infty}\leq1$.
For all other summands (i.e. $1\leq l\leq k-1$), we write
\begin{equation*}
    \begin{bmatrix}
        \alphaa\\ \betaa
    \end{bmatrix}
    \coloneqq \frac{\binom{k}{l}}{(k-1)!\epsilon^{k}}\J_{F}(\y)^{-1}
    \begin{bmatrix}
        \left(\sum_{j}\P_{ij}\phii_{i}^{l}\gammaa_{j}^{k-l}\right)_{i}\\
        \left(\sum_{i}\P_{ij}\phii_{i}^{l}\gammaa_{j}^{k-l}\right)_{j}
    \end{bmatrix}.
\end{equation*}
Using \eqref{eq:jacobian} again, it follows that
\begin{equation*}
    \label{eq:A1}
    \underbrace{\begin{bmatrix}
            \diag(\P\one_{\n}) &\P\\
            \P^{\top} &\diag(\P^{\top}\one_{\n})
    \end{bmatrix}}_{\eqqcolon \A}
    \begin{bmatrix}
        \alphaa\\ \betaa
    \end{bmatrix} 
    = \frac{\binom{k}{l}}{(k-1)!\epsilon^{k-1}}
    \begin{bmatrix*}[l]
        \diag(\P\gammaa^{k-l})\phii^{l}\\
        \diag(\P^{\top}\phii^{l})\gammaa^{k-l}
    \end{bmatrix*},
\end{equation*}
and we aim to estimate $\norm{\alphaa}_{\infty}$ and $\norm{\betaa}_{\infty}$
by $\norm{\phii}_{\infty}$ and $\norm{\gammaa}_{\infty}$. By \cref{lem:jacobian_sym_pos_def}, the matrix $\A$ has a one-dimensional kernel spanned by $\q \coloneqq (\one_{\n}^{\top}, -\one_{\n}^{\top})^{\top}$, and a solution $(\alphaa, \betaa)$ in the orthogonal complement is also a solution to
\begin{equation*}
    \underbrace{(\A + \Delta\q\q^{\top})}_{\B}
    \begin{bmatrix}
        \alphaa\\ \betaa
    \end{bmatrix}
    =\frac{\binom{k}{l}}{(k-1)!\epsilon^{k-1}}
    \begin{bmatrix*}[l]
        \diag(\P\gammaa^{k-l})\phii^{l}\\
        \diag(\P^{\top}\phii^{l})\gammaa^{k-l}
    \end{bmatrix*}
\end{equation*}
for any $\Delta > 0$.
From \textcite{Varah1975}, we know that the $\ell_{\infty}$-norm of the inverse matrix of $\B$ is estimated by
\begin{equation}
    \textstyle
    \label{eq:A2}
    \norm{\B^{-1}}_{\infty} \leq \left[\min_{i}\left( \abs{\B_{ii}} - \sum_{j\neq i}\abs{\B_{ij}}  \right)\right]^{-1}.
\end{equation}
In this case, we calculate for $i = 1,\dots,\n$ that
\begin{equation*}
    \abs{\B_{ii}} - \sum_{\substack{1\leq j\leq 2\n\\ j\neq i}}\abs{\B_{ij}}
    = \sum\limits_{1\leq j\leq \n}\P_{ij} + \Delta - (\n - 1)\Delta - \sum\limits_{1\leq j\leq\n}\abs{\P_{ij} -\Delta}.
\end{equation*}
For any $\Delta \leq \min_{j}\P_{ij}$, this leads to
\begin{equation*}
    \abs{\B_{ii}} - \sum_{\substack{1\leq j\leq 2\n\\ j\neq i}}\abs{\B_{ij}} \leq 2\Delta.
\end{equation*}
Similarly, we get that
\begin{equation*}
    \abs{\B_{\n + i, \n + i}} - \sum_{\substack{1\leq j\leq 2\n\\ j\neq \n + i}}\abs{\B_{n + i, j}} \leq 2\Delta.
\end{equation*}
Choosing $\Delta \coloneqq \min_{ij}\P_{ij}$, we thus obtain that
\begin{equation*}
    \norm{\B^{-1}}_{\infty} \leq \left[2\min_{ij}\P_{ij}\right]^{-1}.
\end{equation*}
Using that
\begin{equation*}
    \norm{\diag(\P^{\top}\phii^l)\gammaa^{k-l}}\leq\norm{\P^{\top}\one_{n}}_{\infty}\norm{\phii}_{\infty}^{l}\norm{\gammaa}_{\infty}^{k-l},
\end{equation*}
and similarly for $\diag(\P\gammaa^{k-l})\phii^l$, finally gives
\begin{equation*}
    \norm{(\alphaa, \betaa)}_{\infty} \leq \frac{\binom{k}{l}}{(k-1)!\epsilon^{k-1}}M \norm{\phii}_{\infty}^{l}\norm{\gammaa}_{\infty}^{k-l}\label{eq:deuflhard-est-other-l}
\end{equation*}
with
\begin{equation*}
    M = \frac{\max\{\norm{\P\one_{n}}_{\infty}, \, \norm{\P^{\top}\one_{n}}_{\infty}\}}{2\min_{ij}\P_{ij}}.
\end{equation*}
Hence we obtain
\begin{multline*}
    \left\|\sum\limits_{k=2}^{\infty}\sum\limits_{l=1}^{k-1}\binom{k}{l} \frac{1}{(k-1)!\epsilon^{k}} \J_{F}(\y)^{-1}
    \begin{bmatrix}
        \left(\sum_{j}\P_{ij}\phii_{i}^l\gammaa_{j}^{k-l}\right)_{i}\\
        \left(\sum_{i}\P_{ij}\phii_{i}^l\gammaa_{j}^{k-l}\right)_{j}    
    \end{bmatrix}\right\|_{\infty}\\
    \begin{aligned}
        &\leq\ M\sum\limits_{k=2}^{\infty}\sum\limits_{l=1}^{k-1}\binom{k}{l} \frac{1}{(k-1)!\epsilon^{k-1}}\norm{\phii}_{\infty}^{l}\norm{\gammaa}_{\infty}^{k-l}\\
        &=\ M\bigg[\left(\e^{(\norm{\phii}_{\infty}+\norm{\gammaa}_{\infty})/\epsilon}-1\right)(\norm{\phii}_{\infty}+\norm{\gammaa}_{\infty})\\
        &\quad-\left(\e^{\norm{\phii}_{\infty}/\epsilon}-1\right)\norm{\phii}_{\infty}-\left(\e^{\norm{\gammaa}_{\infty}/\epsilon}-1\right)\norm{\gammaa}_{\infty}\bigg]\\
        &\leq\ 2\e^{\frac1\epsilon}M(\e^{\frac1\epsilon}-1)\norm{(\phii,\gammaa)}_{\infty}^2
    \end{aligned}
\end{multline*}
for all $\norm{(\phii,\gammaa)}_{\infty}\leq1$.
In summary we obtain
\begin{equation*}
    \norm{\J_{F}(\y)^{-1}[\J_{F}(\y) - \J_{F}(\etaa)](\y - \etaa)}_{\infty}
    \leq(1+2\e^{\frac1\epsilon}M)(\e^{\frac1\epsilon}-1)\norm{\y - \etaa}_{\infty}^2,
\end{equation*}
as desired

To generalize this to the case $\m\neq \n$, one can take $\q = (\Delta_{1}\one_{\n}, -\Delta_{2}\one_{\m})$ for $\Delta_{1}\neq -\Delta_{2}$ and $\Delta_{1}\Delta_{2} = \Delta$ and choose $\Delta_{1}$ to equilibrate the lower bounds for the first $\m$ and the last $\n$ rows of $\B$.

\section{Conclusion}
\label{sec:conclusion}

We have proposed a Newton iteration to solve the entropically regularized discrete optimal transport problem. Different from related Newton type approaches for matrix balancing, our method iterates on the logarithm of the scalings, which seems to be necessary for robust convergence in the optimal transport setting. Numerical examples show that our algorithm is a robust and efficient alternative to the more commonly used Sinkhorn--Knopp algorithm, at least for small regularization strength.

\printbibliography

\end{document}